\documentclass{amsart}

\usepackage[latin1]{inputenc}
\usepackage{amsmath, amsthm, amssymb, booktabs, multirow, graphicx,
  bm, float}

\usepackage{mathtools}
\usepackage{enumitem}
\usepackage{dialogue}
\usepackage{bbold} 
\usepackage[stable]{footmisc}
\usepackage{graphicx}
\usepackage{tikz,nicefrac}
\usetikzlibrary{calc,matrix,arrows,decorations.markings}
\usepackage[colorlinks]{hyperref}
\usetikzlibrary{
  graphs,
  graphs.standard
}
\usepackage{subcaption}

\newtheorem{defin}{Definition}[section]

\newtheorem{theorem}[defin]{Theorem}

\newtheorem{lemma}[defin]{Lemma}
\newtheorem{corollary}[defin]{Corollary}

\newtheorem{remark}[defin]{Remark}

\newcommand{\C}{\mathbb{C}}
\newcommand{\R}{\mathbb{R}}

\newcommand{\Z}{\mathbb{Z}}

\newcommand{\N}{\mathbb{N}}

\makeatletter
\newcommand{\bigperp}{%
  \mathop{\mathpalette\bigp@rp\relax}%
  \displaylimits
}

\newcommand{\bigp@rp}[2]{%
  \vcenter{
    \m@th\hbox{\scalebox{\ifx#1\displaystyle2.1\else1.5\fi}{$#1\perp$}}
  }%
}
\makeatother

\DeclareMathOperator{\vol}{vol}

\DeclareMathOperator{\Tr}{Tr}

\DeclareMathOperator{\supp}{supp}

\DeclareMathOperator{\dist}{dist}

\definecolor{arne}{rgb}{0.2,.7,0.9}

\definecolor{frank}{rgb}{0.8,0.2,0.5}

\definecolor{moritz}{rgb}{0,0.5,0.5}

\begin{document}

\title{A semidefinite program for least distortion embeddings of flat tori into Hilbert spaces}

\address{A.~Heimendahl, M.~L\"ucke, F.~Vallentin,
  M.C.~Zimmermann, Department Mathematik/Informatik, Abteilung
  Mathematik, Universit\"at zu K\"oln, Weyertal~86--90, 50931 K\"oln,
  Germany}

\author{Arne Heimendahl}
\email{arne.heimendahl@uni-koeln.de}

\author{Moritz L\"ucke}
\email{moritz.luecke1997@gmail.com}

\author{Frank Vallentin}
\email{frank.vallentin@uni-koeln.de}

\author{Marc Christian Zimmermann}
\email{marc.christian.zimmermann@gmail.com}

\date{March 30, 2026}


\begin{abstract}
  We derive and analyze an infinite-dimensional semidefinite program
  which computes least distortion embeddings of flat tori $\R^n/L$,
  where $L$ is an $n$-dimensional lattice, into Hilbert spaces.

  This enables us to provide a constant factor improvement over the
  previously best lower bound on the minimal distortion of an
  embedding of an $n$-dimensional flat torus.

  As further applications we prove that every $n$-dimensional flat
  torus has a finite dimensional least distortion embedding, that the
  standard embedding of the standard torus is optimal, and we
  determine least distortion embeddings of all $2$-dimensional flat
  tori.
\end{abstract}

\maketitle

\markboth{A. Heimendahl, M. L\"ucke, F. Vallentin, and
  M.C. Zimmermann}{A semidefinite program for least distortion embeddings of flat tori into Hilbert spaces}

\section{Introduction}

Lattices (discrete subgroups of $n$-dimensional Euclidean spaces) are
central to the geometry of integer programming. One good way to
describe geometric properties of a given lattice $L$  are fundamental
domains of $\R^n$ with respect to translations by $L$. The quotient
$\R^n / L$ is itself a metric space; a flat torus.

Approximating metric spaces by Euclidean spaces to design efficient
approximation algorithms has been a central theme in theoretical
computer science in the last two decades. There the starting point is
computing a least distortion embedding of a ``difficult'' metric
spaces into an ``easy'' normed space.

Least distortion embeddings of flat tori into Hilbert spaces were
first studied by Khot and Naor \cite{Khot2006a} in 2006. One
motivation is that studying the Euclidean distortion of flat tori
might have applications to the complexity of lattice problems, like
the closest vector problem, and might also lead to more efficient
algorithms for lattice problems through the use of least distortion
embeddings. Another motivation comes from comparing the Riemannian
setting to the bi-Lipschitz setting we are discussing here.  On the
one hand, by the Nash embedding theorem, flat tori can be embedded
isometrically as Riemannian submanifolds into Euclidean space; we
refer to \cite{Borrelli2012a} for spectacular visualizations of such
an isometric embedding in the case of the two-dimensional square flat
torus. On the other hand, Khot and Naor showed that flat tori can be
highly non-Euclidean in the bi-Lipschitz setting.

\subsection{Notation and review of the relevant literature}

We review the relevant results of the literature which appeared since
the pioneering work of Khot and Naor. At the same time we set the
notation for this paper.

By $|\cdot|$ we denote the standard norm of $\R^n$
given by $|x| = \sqrt{x^{\sf T} x}$. An $n$-dimensional lattice is a discrete subgroup of $(\R^n, +)$
consisting of all integral linear combinations of a basis of
$\R^n$. Then, a flat torus is the metric space given by the quotient $\R^n / L$ with
some $n$-dimensional lattice $L \subseteq \R^n$ and with metric
\[
  d_{\R^n / L}(x,y) = \min_{v \in L} |x - y - v|.
\]

A Euclidean embedding of $\R^n / L$ is an injective function
$\varphi \colon \R^n / L \to H$ mapping the flat torus $\R^n / L$ into
some (complex) Hilbert space $H$. The \emph{distortion} of $\varphi$ is
\[
  \dist(\varphi) =
  \sup_{\substack{x,y \in \R^n / L \\  x \neq y}} \frac{\|\varphi(x) - \varphi(y)\|}{d_{\R^n/L}(x,y)}
  \cdot \sup_{\substack{x,y \in \R^n / L \\  x \neq y}} \frac{d_{\R^n/L}(x,y)}{\|\varphi(x) - \varphi(y)\|},
\]
where $\| \cdot \|$ is the norm of the Hilbert space $H$.  Here the
first supremum is called the \emph{expansion} of $\varphi$ and the second
supremum is the \emph{contraction} of $\varphi$. When we minimize the
distortion of $\varphi$ over all possible embeddings of $\R^n / L$
into Hilbert spaces we speak of the \emph{least (Euclidean) distortion} of the
flat torus; it is denoted by
 \[
   c_2(\R^n / L) = \inf\{\dist(\varphi) \; : \; \varphi \colon \R^n /
   L \to H \text{ for some Hilbert space } H, \; \varphi \text{ injective}\}.
 \]
Similarly one can define $c_1(\R^n / L)$ by replacing the Hilbert space
by some $L^1$ space.

Khot and Naor showed (see \cite[Corollary 4]{Khot2006a}) that flat
tori can be highly non-Euclidean in the sense that there is a family
of flat tori $\R^n / L_n$ with
\begin{equation} \label{eq:non-Euclidean}
  c_2(\R^n / L_n) = \Omega(\sqrt{n}).
\end{equation}
On the other hand, they noticed (see \cite[Remark 5]{Khot2006a}) that
the standard embedding of the flat torus $\R^n / \Z^n$ into $\R^{2n}$ has distortion $O(1)$\footnote{In fact, we have $ \dist(\varphi) = \pi/2  $ and $ \varphi  $ is an optimal embedding, see Theorem~\ref{thm:standard-torus-embedding}.}. The standard embedding is given
by
\begin{equation}
  \label{eq:standard-embedding}
  \varphi(x_1, \ldots, x_n) = (\cos 2\pi x_1, \sin 2\pi x_1, \ldots,
  \cos 2\pi x_n, \sin 2\pi x_n).
\end{equation}

In fact, Khot and Naor are mainly concerned with bounding
$c_1(\R^n / L)$, which immediately provides bounds for $c_2(\R^n / L)$
because $c_1(\R^n / L) \leq c_2(\R^n / L)$. To state their main
result, leading to \eqref{eq:non-Euclidean}, we make use of the
Voronoi cell of $L$, which is an $n$-dimensional polytope defined as
\[
  V(L) = \{x \in \R^n : |x| \leq |x - v| \text{ for all  } v \in L\}.
\]
The Voronoi cell is a fundamental domain of $\mathbb{R}^n / L$ under
the action of $L$. We denote the volume of $V(L)$ by $\vol L$. Clearly, $|x| =   d_{\R^n / L}(x,0)$ for all $x \in V(L)$.
The covering radius of $L$ is $\mu(L) = \max\{|x| : x \in V(L)\}$,
which is the circumradius of $V(L)$.  The length of a shortest vector
of $L$ is $\lambda(L) = \min\{|v| : v \in L \setminus \{0\}\}$ that
is two times the inradius of $V(L)$. Now the main result (see
\cite[Theorem 5]{Khot2006a}) is
\begin{equation}
\label{eq:KNmain}
c_1(\R^n / L) = \Omega\left( \frac{\lambda(L^*) \sqrt{n}}{\mu(L^*)} \right).
\end{equation}
Here, as usual,
$L^* = \{ u \in \R^n : u^{\sf T} v \in \Z \text{ for all } v \in L\}$
denotes the dual lattice of $L$. They also give an alternative proof
of their main result for $c_2(\R^n / L)$ (see \cite[Lemma
11]{Khot2006a}). The main result leads to the lower
bound \eqref{eq:non-Euclidean} when plugging in duals of lattices
which simultaneously provide dense packings and economical
coverings. Such a family of lattices exist by a theorem of
Butler~\cite{Butler1972a}.

Using Korkine-Zolotarev reduction Khot and Naor determine an embedding
of $\R^n / L$ into $\R^{2n}$ with distortion $O(n^{3n/2})$ (see
\cite[Theorem 6]{Khot2006a}).

 Haviv and Regev \cite[Theorem 1.3]{Haviv2010} found an improved
 embedding that yields $c_2(\R^n/L) = O(n \sqrt{\log n})$. They also
 improved on \eqref{eq:KNmain} and showed in \cite[Theorem
 1.5]{Haviv2010} that for any $n$-dimensional lattice $L$ we have
 \begin{equation}
   \label{eq:HRbound}
  c_2(\R^n/L) \geq \frac{\lambda(L^*) \mu(L)}{4\sqrt{n}},
\end{equation}
which improves on \eqref{eq:KNmain} because
$\mu(L) \mu(L^*) \geq \Omega(n)$ holds for every $n$-dimensional
lattice. This follows from a simple volume argument giving
$\mu(L) = \Omega(\sqrt{n} (\vol L)^{1/n})$ and $ \vol L^* = (\vol L)^{-1}  $.

Recently, Agarwal, Regev, Tang \cite{Agarwal2020a} constructed
excellent embeddings of flat tori having low distortion and showed that
the lower bound~\eqref{eq:non-Euclidean} is nearly tight: For every
lattice $L \subseteq \R^n$ there exists an embedding of $\R^n/L$ into
Hilbert space with distortion $O(\sqrt{n \log n})$.

\subsection{Aim and method}

In this paper we want to add a semidefinite optimization perspective
to this story.

For finite metric spaces it is known that one can compute least
distortion Euclidean embeddings via a semidefinite program (SDP),
which is linear optimization over the cone of positive semidefinite
matrices. We want to extend this result from finite metric spaces to
flat tori. This will yield, via semidefinite programming duality, an
algorithmic method for proving nonembeddability results. In
particular, this leads to a new, simple proof of~\eqref{eq:HRbound}.
In fact we even get a constant factor improvement that is tight in the case of the standard torus.

First we recall the semidefinite program for finding least Euclidean distortion
 embeddings of finite metric spaces.  Suppose we consider a
finite metric space $X$ with distance function $d$. Then, as first
observed by Linial, London, Rabinovich \cite{Linial1995a}, we can
find a least distortion embedding of $(X,d)$ into a Hilbert space
algorithmically by solving the following semidefinite program
\begin{equation}
  \label{eq:sdp-primal-low-distortion}
  \begin{split}
 \min\{C \; : \; & C \in \R_+, Q \in \mathcal{S}^X_+,\\
&    d(x,y)^2 \leq Q_{xx} - 2Q_{xy} +
Q_{yy} \leq C d(x,y)^2 \text{ for all } x, y \in X\},
\end{split}
\end{equation}
where $\mathcal{S}^X_+$ denotes the convex cone of positive
semidefinite matrices whose rows and columns are indexed by the
elements of $X$. The optimal solution $C$ of this semidefinite program
equals $c_2(X,d)^2$ and if $Q$ attains the optimal solution, then we
can determine a least distortion embedding $\varphi : X \to \R^X$ with
the property $\varphi(x) \cdot \varphi(y) = Q_{xy}$ by considering a
Cholesky decomposition of $Q$.

This shows how to compute (in fact in polynomial time) an optimal Euclidean
embedding of a finite metric space. Another benefit of this
formulation is that we can apply duality theory of semidefinite
programs. Then the dual maximization problem will play a key role to
determine lower bounds for $c_2(X,d)$.  By using strong duality we
arrive at the following result: The least distortion of a finite
metric space $(X,d)$, with $X = \{x_1, \ldots, x_n\}$, into Euclidean
space is given by
\begin{equation}
  \label{eq:sdp-dual-low-distortion}
  c_2(X,d)^2 = \max \left\{
\frac{\sum_{i,j=1 : Y_{ij} > 0}^n Y_{ij} d(x_i,x_j)^2}{-\sum_{i,j=1 :
    Y_{ij} < 0}^n Y_{ij} d(x_i,x_j)^2} :
Y \in \mathcal{S}^n_+, Y\mathbf{e} = 0
\right\}.
\end{equation}
The condition $Y\mathbf{e} = 0$ says that the all-ones vector
$\mathbf{e}$ lies in the kernel of $Y$. A proof of this result is
detailed in Matou\v{s}ek \cite{Matousek2002} or in Laurent, Vallentin \cite{LaurentV}.

This lower bound has been extensively used to determine the least
distortion Euclidean embeddings of the shortest path metric of several
graph classes. Linial, Magen~\cite{LinialM2000} computed least
distortion embeddings of products of cycles and of expander graphs.
Least distortion Euclidean embeddings of strongly regular graphs and
of more general distance regular graphs were first considered by
Vallentin~\cite{Vallentin2008b}. This was further extended by
Kobayashi, Kondo \cite{KobayashiK2015}, Cioab\u{a}, Gupta, Ihringer,
Kurihara \cite{CioabaGIK2021}. Linial, Magen, Naor
\cite{LinialMN2002} considered graphs of high girth using this
approach.

To apply the bound~\eqref{eq:sdp-dual-low-distortion} one has to
construct a matrix $Y$, which sometimes appears to come out of the
blue. By complementary slackness, which is the same as analyzing the
case of equality in the proof of weak duality, we get hints where to
search for an appropriate matrix $Y$: If $Y$ is an optimal solution of the maximization
problem~\eqref{eq:sdp-dual-low-distortion}, then $Y_{ij} > 0$ only for
the \emph{most contracted pairs}. These are pairs $(x_i, x_j)$ for
which $\frac{d(x_i,x_j)}{\|f(x_i) - f(x_j)\|}$ is
maximized. Similarly, then $Y_{ij} < 0$ only for the \emph{most
  expanded pairs}, maximizing
$\frac{\|f(x_i) - f(x_j)\|}{d(x_i,x_j)}$.

Linial, Magen~\cite{LinialM2000} realized that for graphs most
expanded pairs are simply adjacent vertices. However, most contracted
pairs are more mysterious and there is no characterization known. The
first intuition that the largest contraction occurs at pairs at maximum
distance is wrong in general.

\subsection{Contribution and structure of the paper}

In Section~\ref{sec:SDP} of this paper we derive a new
infinite-dimensional semidefinite program for determining a least
distortion embedding of flat tori into Hilbert spaces which is
analogous to~\eqref{eq:sdp-primal-low-distortion}. It is given in
Theorem~\ref{thm:primal-sym} where we additionally apply symmetry
reduction techniques in the spirit of~\cite{BachocGSV2012} to reduce
the original infinite-dimensional SDP into an infinite-dimensional
linear program that involves Fourier analysis. Then we realize that in
a Euclidean embedding of a flat torus there are no most expanded
pairs: The expansion is only attained in the limit by pairs whose
distance tends to zero. This is in perfect analogy to the graph case
where the most expanded pairs are also attained at minimal
distance. This insight has the advantage that in the infinite
dimensional linear program some of the infinitely many constraints can
be replaced by only one finite-dimensional semidefinite
constraint. This is the content of
Theorem~\ref{thm:primal-simple}. Its dual program is derived in
Theorem~\ref{thm:dual-simple} which is analogous
to~\eqref{eq:sdp-dual-low-distortion}.

In Section~\ref{sec:properties} we further investigate the properties
of the optimization problems given in Theorem~\ref{thm:primal-simple}
and Theorem~\ref{thm:dual-simple}. These properties will be
used in the next sections.

In the last sections we apply our new methodology. In
Section~\ref{sec:finite-dimensional-embedding} we prove that an
$n$-dimensional flat torus always admits a finite dimensional least
distortion em\-bedding, a space of (complex) dimension $2^{n}-1$
suffices. Section~\ref{sec:lower-bound} contains a new and simple
proof of our constant factor improvement of the lower bound given
in~\eqref{eq:HRbound}. In Section~\ref{sec:standard-torus} we show
that the standard embedding~\eqref{eq:standard-embedding} of the
standard torus is indeed optimal and has distortion $\pi/2$. We give
an optimal embedding of the lattices $D_n^*$ in
Section~\ref{sec:dnstar}. In Section~\ref{sec:2d-embeddings} we
determine least distortion embeddings of all two-dimensional flat
tori. Open questions are discussed in
Section~\ref{sec:open-questions}.

\section{An infinite-dimensional SDP}
\label{sec:SDP}

Starting from~\eqref{eq:sdp-primal-low-distortion} we want to derive a
similar, but now infinite-dimensional, semidefinite program which can
be used to determine $c_2(\R^n/L)$.

\subsection{Primal program}

The first step is to apply classical results on positive definite kernels\footnote{A kernel $Q$ is called positive
definite if and only if, for all $N \in \N$ and for all $x_1, \ldots, x_N \in \R^n / L$, the matrix
$(Q(x_i,x_j))_{1 \leq i,j \leq \N} \in \C^{N \times N}$ is Hermitian
and positive semidefinite. This naming convention is unfortunate but
for historical reasons unavoidable.} (which are attributed to Mercer \cite{Mercer1909a} and Moore \cite{Moore1916a}). This 
enables us to optimize over all embeddings $\varphi : \R^n / L \to H$ into some Hilbert space $H$. 

Let us provide some details about the relation between positive definite kernels and embeddings into Hilbert spaces.

Going from finite distortion embeddings to positive semidefinite kernels is easy: Suppose $\varphi : \R^n / L \to H$ is an embedding into some Hilbert space $H$ with $\dist(\varphi) < \infty$. Then $\varphi$ has finite expansion and, in particular, it is Lipschitz continuous. It follows that the kernel $Q(x,y) = (\varphi(x), \varphi(y))$, where $(\cdot , \cdot)$ denotes the inner product on $H$, is continuous and positive definite.

The converse is deeper: Suppose a continuous and positive kernel $Q : \R^n / L \times \R^n / L \to \C$ is given. Then, by a theorem of Mercer (see for instance Riesz and Sz.-Nagy \cite[$\S$ 98]{Riesz1990a} or Cucker and Smale \cite[Chapter III]{Cucker2001}), there exists a spectral decomposition of the form
\[
Q(x,y) = \sum_{k=1}^\infty \lambda_k \phi_k(x) \overline{\phi_k(y)},
\]
with real numbers $\lambda_k \geq 0$ and functions $\phi_k : \R^n / L \to \C$, where the series converges absolutely for each $x, y \in \R^n / L$ and uniformly on $\R^n / L \times \R^n / L$. Defining 
\[
\varphi : \R^n / L \to \ell^2 \quad \text{by} \quad x \mapsto (\sqrt{\lambda_k} \phi_k(x))_{k=1}^\infty,
\]
we obtain a continuous map into the Hilbert space $\ell^2$ of square summable sequences, which satisfies $Q(x,y) = (\varphi(x), \varphi(y))$ for all $x,y \in \R^n / L$.

Therefore we get
\[
  \begin{split}
  c_2(\R^n/L)^2 = \inf\{C \; : \; & C \in \R_+, Q \text{ continuous and positive definite kernel},\\
  &  \begin{array}{rcl} d_{\R^n/L}(x,y)^2 & \leq &  Q(x,x) - 2 \Re(Q(x,y)) + Q(y,y) \\
        & \leq & C d_{\R^n/L}(x,y)^2 \; \text{ for all } x,y \in \R^n/L\}.
   \end{array}
  \end{split}
\]
Here we scaled the embedding $\varphi$ which is defined through $Q$ so
that the contraction of $\varphi$ equals $1$. The real part $\Re(Q)$
of a positive definite kernel is positive definite again and we can
restrict to real-valued positive
definite kernels for determining $c_2(\R^n/L)$.

For the second step we apply a standard group averaging argument. If
$Q$ is a feasible solution for the minimization problem above, so is
its group average
\[
  \overline{Q}(x,y) = \frac{1}{\vol(\R^n/L)} \int_{\R^n/L} Q(x-z, y-z)
  \, dz.
\]
Thus, instead of minimizing over
continuous and positive definite kernels $Q$ it suffices to minimize over continuous,
real functions $f \colon \R^n/L \to \R$ which are positive definite,
i.e.\ the kernel $(x,y) \mapsto f(x - y)$ is continuous and positive definite; see
also the proof of Theorem 3.1 in the paper \cite{Aharoni1985a} by
Aharoni, Maurey, Mityagin.

A continuous, positive definite function $f \in L^{\infty}(\R^n/L)$ defines a positive linear functional on $L^1(\R^n/L)$ by
\[
\int\int g(x) \overline{g(y)} f(x-y) \, dx \, dy \geq 0 \quad \text{for all } g \in L^1(\R^n/L). 
\]
So, by the Banach-Alaoglu theorem, see for example Folland~\cite[Chapter 3.3]{Folland1995a}, 
for any $C$, the set 
\[
\{f \in L^{\infty}(\R^n/L) : f \text{ continuous and positive definite}, \|f\|_{\infty} \leq C\}
\]
is weak-$*$ compact. Therefore, a minimum distortion embedding always exists.

Note also that $(x,y) \mapsto d_{\R^n/L}(x,y)^2$ only depends on the
difference $x-y$. So we can replace $(x,y)$ by $(x-y,0)$ and we can
move $x-y$ by a lattice vector translation into the Voronoi cell
$V(L)$. Hence,
\[
  \begin{split}
  c_2(\R^n/L)^2 = \min\{C \; : \; & C \in \R_+, f \colon \R^n/L \to \R \text{
    continuous and positive definite},\\
&  |x|^2 \leq 2(f(0) - f(x)) \leq C
|x|^2 \; \text{ for all } x \in V(L)\}.
  \end{split}
\]

In the third step we parametrize continuous positive definite functions by
the Fourier coefficients using Bochner's theorem, cf.\
Folland~\cite[(4.18)]{Folland1995a}, which says that a continuous
function $f : \R^n / L \to \C$ is positive definite if and only if all its
Fourier coefficients
\[
\widehat{f}(u) = \int_{\R^n/L} f(x) e^{-2\pi i u^{\sf T} x} \, dx,
\]
with $u \in L^*$ are nonnegative and $\widehat{f}$ lies in
\[
  \ell^1(L^*) = \left\{ z \colon L^* \to \C : \sum_{u \in L^*} |z(u)| < \infty \right\}.
\]
Then if $f$ is real, continuous and
of positive definite we have the representation
\[
  f(x) = \sum_{u \in L^*} \widehat{f}(u) e^{2\pi i u^{\sf T} x},
\]
where the convergence is absolute and uniform, with $\widehat{f} \in
\ell^1(L^*)$, $\widehat{f}(u) \geq 0$ and $\widehat{f}(u) =
\widehat{f}(-u)$ for all $u \in L^*$. Thus,
\[
  f(x) = \sum_{u \in L^* } \widehat{f}(u) \cos(2\pi u^{\sf T} x).
\]
Writing $ f $ in this form, one can express $c_2(\R^n / L)^2$ as an infinite-dimensional
linear program:

\begin{theorem}
\label{thm:primal-sym}
  The least distortion Euclidean embedding of a flat torus $\R^n / L$
  is given by
  \begin{equation}
  \label{eq:primal}
  \begin{split}
  c_2(\R^n/L)^2 = \min\big\{C \; : \; & C \in \R_+, z \in \ell^1(L^*),
  z(u) = z(-u) \geq 0 \text{ for all } u \in L^*,\\
&  |x|^2 \leq 2 \sum_{u \in L^*} z(u) (1- \cos(2\pi u^{\sf T} x))\leq C
|x|^2\\
& \quad \text{ for all } x \in V(L)\big\}.
  \end{split}
\end{equation}
A feasible solution of the above minimization problem $(C,z)$
determines a Euclidean embedding $\varphi$ of $\R^n/L$ with distortion
$dist(\varphi) \leq \sqrt{C}$ by
\begin{equation}
  \label{eq:varphi-embedding}
   \varphi : \R^n/L \to \ell^2(L^*), \quad
    x  \mapsto \left(\sqrt{z(u)} e^{2\pi i u^{\sf T} x} \right)_{u \in L^*},
  \end{equation}
  with complex Hilbert space
\[
  \ell^2(L^*) = \left\{z : L^* \to \C : \left(\sum_{u \in L^*} |z(u)|^2 \right)^{1/2} < \infty\right\}.
\]
\end{theorem}

It is worth to mention that the embedding $\varphi$ of
Theorem~\ref{thm:primal-sym} embeds the flat torus $\R^n / L$ into a
direct product of circles
\[
  \prod_{u \in L^*} \sqrt{z(u)} S^1 \quad \text{with} \quad
  \|\varphi(x)\|^2 = \sum_{u \in L^*} z(u) \text{ for all } x \in L.
\]
The support of $z$ contains a basis of $L^*$ since the embedding is injective.
Using the fact $z(u) = z(-u)$ we could also use the real embedding
$\varphi'$ with
\[
  [\varphi'(x)]_u = \sqrt{z(u)} (\cos 2\pi u^{\sf T} x, \sin 2\pi u^{\sf
    T} x),
\]
where $u$ runs through $L^* / \{\pm 1\}$ and which has the same
distortion as $\varphi$.

On the other hand, the constraint $z(u) = z(-u)$ is clearly redundant
in the minimization problem of Theorem~\ref{thm:primal-sym}.

Sometimes, it is convenient to use the identity
\begin{equation}
\label{eq:double-angle}
2 \sum_{u \in L^*} z(u) (1 - \cos(2\pi u^{\sf T} x)) = 4 \sum_{u \in L^*} z(u) \sin(\pi u^{\sf T} x)^2,
\end{equation}
which follows from the cosine double angle formula $1 - \cos(\alpha) = 2\sin(\alpha/2)^2$.

\medskip

Now we want to simplify the infinitely many inequalities
  \begin{equation}
 \label{eq:contraction-inequality}
2  \sum_{u\in L^*} z(u) (1-\cos(2\pi u^{\sf T} x)) \leq C |x|^2 \text{ for all } x \in V(L),
\end{equation}
which occur in~\eqref{eq:primal}, by only \textit{one}
finite-dimensional semidefinite condition. For this we observe that in any embedding there are no most expanded pairs: the
corresponding supremum
$\sup\left\{\frac{\|\varphi(x) -
  \varphi(y)\|}{d_{\R^n/L}(x,y)} : x,y \in \R^n / L,  x \neq y \right\}$ is only attained by a limit of pairs
  whose distance tends to $0$.

\begin{lemma}
  Let $L \subseteq \R^n$ be an $n$-dimensional lattice. Let $(C,z)$ be
  as in~\eqref{eq:primal}. Inequality
  \eqref{eq:contraction-inequality} is satisfied if and only if
\begin{equation}
   \label{eq:contraction-inequality-simple}
    4 \pi^2 \sum_{u\in L^*} z(u) (u^{\sf T} x)^2 \le  C |x|^2 \text{ for all }  x \in \R^n.
\end{equation}
\end{lemma}

Note that \eqref{eq:contraction-inequality-simple} holds for all $x \in \R^n$.

\begin{proof}
Using identity \eqref{eq:double-angle} and by the inequality $ |\sin(\alpha)| \leq
|\alpha|$ we have
  \[
     4 \sum_{u\in L^*} z(u) \sin(\pi u^{\sf T} x) \leq 4\pi^2
      \sum_{u \in L^*} z(u) (u^{\sf T} x)^2.
 \]
Thus, \eqref{eq:contraction-inequality-simple} implies
 \eqref{eq:contraction-inequality}.

Conversely, assume that~\eqref{eq:contraction-inequality-simple}
is not satisfied. There exists $x^* \in \R^n$ with
\begin{equation*}
  4 \pi^2 \sum\limits_{u\in L^*} z(u) (u^{\sf T} x^*)^2 >  C|x^*|^2.
\end{equation*}
For $r \geq 0$ define the function
\[
  f(r) =   2 \sum_{u\in L^*} z(u) (1-\cos(2\pi u^{\sf T} (rx^*))) - C|rx^*|^2
\]
and consider its Taylor expansion
\[
  f(r) = \left( 4\pi^2 \sum\limits_{u\in L^*} z(u) (u^{\sf T}
    x^*)^2  - C|x^*|^2\right) r^2 + o(r^2).
\]
Writing $ f $ this way and using the assumption, $f(r)$ is positive for
sufficiently small $r$. Thus, \eqref{eq:contraction-inequality} is not
satisfied.
\end{proof}

Inequality \eqref{eq:contraction-inequality-simple} can also be
rewritten as an inequality of the largest eigenvalue $\lambda_{\max}$
of a corresponding matrix
\[
  \lambda_{\max} \left(   4 \pi^2 \sum\limits_{u\in L^*} z(u) uu^{\sf
      T} \right) \leq C
\]
or equivalently as a semidefinite condition
\[
  C I - 4 \pi^2 \sum\limits_{u\in L^*} z(u) uu^{\sf
      T}  \in \mathcal{S}_+^n,
  \]
  where $I$ denotes the identity matrix. With this lemma we
  arrive at the following simplification of~\eqref{eq:primal}.

  \begin{theorem}
  \label{thm:primal-simple}
  The least distortion Euclidean embedding of a flat torus $\R^n / L$
  is given by
  \begin{equation}
  \label{eq:primal-simple}
  \begin{split}
  c_2(\R^n/L)^2 = \min\big\{C \; : \; & C \in \R_+, z \in \ell^1(L^*),
  z(u) = z(-u) \geq 0 \text{ for all } u \in L^*,\\
  &  |x|^2 \leq 2 \sum_{u \in L^*} z(u) (1- \cos(2\pi u^{\sf T} x))
  \text{ for all } x \in V(L),\\
  & CI - 4\pi^2 \sum_{u \in L^*} z(u) uu^{\sf T}\in \mathcal{S}^n_+ \big\}.
  \end{split}
\end{equation}
\end{theorem}

\subsection{Dual program}

We derive the dual of~\eqref{eq:primal-simple} to systematically find
lower bounds for $c_2(\R^n / L)$.

\begin{theorem}
\label{thm:dual-simple}
Suppose that $(C, z)$ is feasible
for~\eqref{eq:primal-simple}, then
\begin{equation}
\label{eq:dual}
\begin{split}
C \geq c_2(\R^n / L)^2 \geq \sup \big\{ & 2 \pi^2 \int_{V(L)} |x|^2 \, d\nu(x) \; : \\
& \quad \nu \in \mathcal{M}_+(V(L)), Y \in \mathcal{S}^n_+, \Tr(Y) = 1,\\
& \quad \int_{V(L)} (1-\cos(2\pi u^{\sf T} x)) \, d\nu(x) \leq u^{\sf T}
Y u\\
& \qquad \text{ for all } u \in L^*  \big\},
\end{split}
\end{equation}
where $\mathcal{M}_+(V(L))$ is the cone of Borel measures supported on
$V(L)$. In~\eqref{eq:dual} equality holds for a feasible $(\nu,Y)$
if and only if
\[
  \left(C I - 4 \pi^2 \sum_{u\in L^*}z(u) uu^{\sf T}\right) Y =
   0,
\]
and the measure $\nu$ is only supported on vectors $x \in V(L)$ for which equality
\[
  |x|^2 = 2 \sum_{u \in L^*} z(u) (1- \cos(2\pi u^{\sf T} x))
\]
holds, and for all vectors $u \in L^*$ with $z(u) \neq 0$ we have
\[
\int_{V(L)} (1-\cos(2\pi u^{\sf T} x)) \, d\nu(x) = u^{\sf T}
   Y u.
 \]
\end{theorem}
\begin{proof}
For two symmetric matrices $ A,B $ we define $ \langle A, B \rangle = \Tr(AB) $.
 Using the feasibility of $(C,z)$ and $(\nu, Y)$ we get
\begin{align*}
   & C - 2\pi^2 \int_{V(L)} |x|^2 \, d\nu(x) \\
   \geq \; & \left\langle 4 \pi^2 \sum_{u \in L^*} z(u) u u^{\sf T}, Y
             \right\rangle - 4\pi^2 \int_{V(L)} \sum_{u\in L^*} z(u) (1-\cos(2\pi u^{\sf T} x)) \, d\nu (x) \\
   = \; & 4 \pi^2 \sum_{u \in L^*} z(u) \left( \langle uu^{\sf T}, Y \rangle -
          \int_{V(L)} (1 - \cos(2\pi u^{\sf T} x)) \, d \nu(x) \right)\\
  \geq \; & 0.
\end{align*}
When analyzing the case of equality we find the three conditions of
the theorem.
\end{proof}

\begin{remark}
As a side note, we would like to mention that strong duality between \eqref{eq:primal-simple} and \eqref{eq:dual} holds. In particular, the supremum in \eqref{eq:dual} is attained. This can be derived using \cite[(7.2) Theorem]{Barvinok2002} (taking \eqref{eq:dual} as the primal problem in canonical form) in combination with \cite[(7.3) Lemma]{Barvinok2002}.
\end{remark}

\section{Properties and observations}
\label{sec:properties}

We collect some results that are consequences of the primal and dual
formulation of the preceding section, including some auxiliary results
used in later sections.

\subsection{Subquadratic inequality}

First, we show that the functions of the form
\begin{equation}
\label{eq:f-function}
  f(x) = 2 \sum_{u \in L^*} z(u) (1 - \cos(2\pi u^{\sf T} x)) \text{
    with } z(u) \geq 0
\end{equation}
are subquadratic, this auxiliary result is going to be used a number of times. Note that we have
\[
  f(x-y) = \|\varphi(x) - \varphi(y)\|^2
\]
for the embedding $\varphi$ in \eqref{eq:varphi-embedding}. Suppose
for a moment that $\varphi$ was an
isometry, then $f$ would satisfy the parallelogram law
\[
  f(x-y) + f(x+y) = 2 f(x) + 2 f(y)
\]
and it would be a homogeneous quadratic form
\[
  f(\lambda x) = \lambda^2 f(x).
\]
However, $\varphi$ cannot be a Hilbert space isometry, but the next
lemma shows that we have at least two inequalities.

\begin{lemma}
  \label{lemma:sub1}
  The function $f$ defined in \eqref{eq:f-function} is subquadratic, i.e.\ it
  satisfies
  \begin{equation}
\label{eq:subquadratic-inequality}
 f(x+y) + f(x-y) \leq 2f(x) + 2f(y) \quad \text{for all } x,y\in \R^n.
\end{equation}
Furthermore,
\begin{equation}
  \label{eq:subquadratic-inequality2}
    f(\lambda x) \le \lambda^2f(x) \quad  \text{for all } \lambda \in
    \N, x\in \R.
\end{equation}
We have equality in \eqref{eq:subquadratic-inequality} if and only for every $u \in \supp(z)$, $u^{\sf T}x \in \Z$ or $u^{\sf T} y \in \Z$.
\end{lemma}

A proof for \eqref{eq:subquadratic-inequality2}, which holds for any subquadratic function, can also be found in
\cite{Kominek2006a}; we provide it here for the convenience of the
reader.

\begin{proof}
To show that $f(x) = 4\sum_{u \in L^*} z(u) \sin(\pi u^{\sf T} x)^2$ is subquadratic it suffices to prove the
inequality
\[
\sin(\alpha + \beta)^2 + \sin(\alpha - \beta)^2 \leq 2 \sin(\alpha)^2 + 2 \sin(\beta)^2
\]
for all $\alpha, \beta \in \R$. This is elementary by the sine
addition formula:
\[
\begin{split}
\sin(\alpha + \beta)^2 + \sin(\alpha - \beta)^2 
& = (\sin \alpha \cos \beta + \sin \beta \cos \alpha)^2 + (\sin \alpha \cos \beta - \sin \beta \cos \alpha)^2\\
& = 2(\sin \alpha)^2 (\cos \beta)^2 + 2(\sin \beta)^2 (\cos \alpha)^2 \\
& \leq 2 \sin(\alpha)^2 + 2 \sin(\beta)^2,
\end{split}
\]
where equality holds if and only if $\alpha = \pi u^{\sf T} x$ or $\beta = \pi u^{\sf T} y$ is an integral multiple
of $\pi$.

For even $\lambda$ we directly use~\eqref{eq:subquadratic-inequality}
\[
\begin{split}
  f(\lambda x) & = \, f\left(\frac{\lambda}{2} x + \frac{\lambda}{2} x\right) + f\left(\frac{\lambda}{2} x - \frac{\lambda}{2} x\right)\\
  & \, \leq 4 \left(\frac{\lambda}{2}\right)^2 f(x) = \lambda^2 f(x)
\end{split}
\]
since $f(0) = 0$. For odd $\lambda \geq 3$ we
use~\eqref{eq:subquadratic-inequality} and proceed by induction
\[
\begin{split}
  f(\lambda x) + f(x) & = \, f\left( \left(\frac{\lambda - 1}{2} + 1 \right)
    x + \frac{\lambda - 1}{2} x \right) + f\left( \left(\frac{\lambda - 1}{2} + 1 \right)
      x - \frac{\lambda - 1}{2} x \right)\\
      & \leq \, 2 f\left( \left(\frac{\lambda - 1}{2} + 1 \right)
        x\right) + 2 f \left(\frac{\lambda - 1}{2} x \right)\\
      & \leq \, 2 \left( \left(\frac{\lambda - 1}{2} + 1 \right)^2 +
        \left(\frac{\lambda - 1}{2}\right)^2 \right) f(x)\\
      & = \, \lambda^2 f(x) + f(x).
\end{split}
 \]
\end{proof}

\subsection{Dual feasibility}

In general the dual program \eqref{eq:dual} has infinitely many conditions of
the form
\begin{align} \label{eq:dual-Y-cond}
	\int_{V(L)} (1-\cos(2\pi v^{\sf T}x))  \, d\nu(x) \le \Tr(vv^{\sf T} Y), \qquad v \in L^*.
\end{align}
We will now show that sometimes already finitely many constraints are
sufficient to imply all conditions~\eqref{eq:dual-Y-cond}.  The first
observation is the following:

\begin{lemma}\label{lemma:dual-simplifiation-by-ineqs}
  Let $ q_a(x) = 1-\cos(2\pi a^{\sf T}x) $. The (in-)equalities
	\begin{align}\label{eq:uvY-condition}
		\int_{V(L)}  q_{a}(x)  \, d\nu(x)  &\le
                                                           \Tr(aa^{\sf T}
                                                           Y), \quad
                                                           \int_{V(L)}
                                                           q_{b}(x) \,
                                                           d\nu(x)
                                                           \le
                                                           \Tr(bb^{\sf T}
                                                           Y),   \\
		\label{eq:pmY-condition}
		\int_{V(L)}  q_{a - b}(x)  \, d\nu(x) &= \Tr((a- b)(a- b)^\mathsf{T}Y)
	\end{align}
	imply
	\begin{align*}
		\int_{V(L)}  q_{a + b}(x)  \, d\nu(x) \le\Tr((a + b)(a + b)^\mathsf{T}Y) .
	\end{align*}
        The corresponding result also holds when $q_{a-b}$ and $q_{a+b}$
        are interchanged.
      \end{lemma}

\begin{proof}
  As shown in the proof of Lemma~\ref{lemma:sub1}, the function $q_a$
  is subquadratic and therefore
	\begin{align*}
		\int_{V(L)}  q_{a+b}(x) \, d\nu(x)  + \int_{V(L)}
          q_{a-b}(x)  \, d\nu(x) &\le 2 \int_{V(L)}  (q_a(x) +q_b(x)) \,
                                   d\nu(x) \\
		&\le 2\Tr(aa^\mathsf{T}Y) + 2 \Tr(bb^\mathsf{T} Y),
	\end{align*}
	which by \eqref{eq:pmY-condition} is equivalent to
	\begin{align*}
			\int_{V(L)}  q_{a+b}(x)  \, d\nu(x)  &\le 2\Tr(aa^\mathsf{T}Y) + 2 \Tr(bb^\mathsf{T} Y) - \Tr((a- b)(a- b)^\mathsf{T}Y)   \\
			& = \Tr((a+b)(a+b)^\mathsf{T} Y).\qedhere
	\end{align*}
\end{proof}

The above lemma can be used to replace the infinitely many constraints
\eqref{eq:dual-Y-cond} by finitely many using the shortest vectors in
cosets of the form $ v+2L^* $ for $ v \in L^* $. This is the content of the next lemma, which is going to be used in the proof of Theorem~\ref{thm:dual-solution-for-2d-case}, where we find the least distortion embedding of two-dimensiona flat tori; it was also used
to determine the exact least Euclidean distortion of $\R^2 / A_2$ and $R^8 / E_8$ in \cite{Moustrou2023a}.

The proof of the next lemma relies on a characterization of \emph{Voronoi
vectors}. These are lattice vectors $ v \in L\setminus\{0\}  $ such that the set $  F_v :=V(L) \cap \{x \, : \, v^\mathsf{T}x = \frac{1}{2}
v^\mathsf{T}v \}$ defines a non-empty face of $ V(L)$. Moreover, $ v \in L $ is called \emph{Voronoi relevant} if $ F_v $ is a facet of $ V(L) $, i.e.\ an $ (n-1) $-dimensional face of $ V(L) $.

An element $ v \in L\setminus\{0\} $ is a Voronoi vector of $ V(L)$ if and
only if $ \pm v$ are shortest vectors in the coset $ v+2L$ and
$ \pm v$ are Voronoi relevant if and only if they are the \emph{only}
shortest vectors in $ v+2L$. For a proof see \cite[Chapter 21, Theorem
10]{Conway1988a} and \cite[Theorem 2]{Conway1992}.

\begin{lemma}
  \label{lem:equality-for-Vor-rel-vecs}
  If \eqref{eq:dual-Y-cond} is tight for at least one shortest vector
  in each coset of the form $ v+2L^*, \, v \in L^*$, then
  \eqref{eq:dual-Y-cond} holds for all $ v \in L^*$.
\end{lemma}

\begin{proof}
Assume that \eqref{eq:dual-Y-cond} is tight for at least one
shortest vector in each coset. We will first prove by
induction that \eqref{eq:dual-Y-cond} also holds for all Voronoi
vectors.

Let $u$ be a Voronoi vector. We perform an induction on the codimension $k = n - \dim(F_u)$ of the face $F_u$ of $V(L^*)$ defined by $u$.

The base case is $k=1$ follows immediately from the assumption because $\pm u$ are the only shortest vectors in the corresponding coset. 

Suppose $k > 1$. Assume that $w \in u + 2L^*$ is a shortest vector such that \eqref{eq:dual-Y-cond} is tight. Furthermore, assume that $u \neq \pm w$ (otherwise we would be done). Define $a = (u+w)/2$ and $b = (u-w)/2$ so that $u = a+b$ and $w = a-b$. Note that $a,b \in L^*$ and that $F_a$, $F_b$ are faces of codimension at most $k-1$ with $F_a \cap F_b = F_u$. Hence, by the induction hypothesis \eqref{eq:dual-Y-cond} holds for $a$ and $b$ and we can apply Lemma~\ref{lemma:dual-simplifiation-by-ineqs} to infer that \eqref{eq:dual-Y-cond} also holds for $u$.

\smallskip

Now assume that $ v $ is a not a Voronoi vector. Then there
exists a shortest vector $ u \in v+2L^* $ for which
\eqref{eq:dual-Y-cond} is tight and $ |u| < |v| $.

Then, again $ \frac{1}{2}(u \pm v) \in L^* $ and as
\begin{align*}
|\frac{1}{2}(u \pm v) | \le \frac{1}{2} (|u| + |v|) < |v|,
\end{align*}
we can argue by an analogous inductive argument (based on the
norm) as before that \eqref{eq:dual-Y-cond} holds for
$ \frac{1}{2}(u \pm v) $.  Finally, we can use
Lemma~\ref{lemma:dual-simplifiation-by-ineqs} to infer that~\eqref{eq:dual-Y-cond} is valid for $ v $ as well.
\end{proof}

\section{Least Euclidean distortion embeddings always have finite
  dimension}
\label{sec:finite-dimensional-embedding}

The goal of this section is to prove that for every $ n $-dimensional
lattice, there always exists a least distortion embedding of
$ \R^n/L $ that is finite-dimensional. In the sense of
Theorem~\ref{thm:primal-sym}, this means that there is always an
optimal solution $ (C,z) $ for~\eqref{eq:primal-simple} such that the
support of $z$ is finite.

Additionally, our arguments will reveal that the constructed optimal
solution with finite support has only support on at most one vector
per coset $ v+2L^* $ of $ L^*/2L^* $ and that $ \supp(z) $ only
contains primitive lattice vectors.  An element $v \in L$ is called
\emph{primitive} for $L$ if $\alpha v \in L$ with $\alpha \in \Z$
implies $\alpha = \pm 1$.

The first step towards proving that there is always a finite-dimensional least Euclidean distortion embedding is the following observation.
\begin{lemma}\label{lemma:support-decrease-simple}
	Assume that $ (C,z) $ is a solution for~\eqref{eq:primal-simple}.
	\begin{enumerate}
		\item  If there are $ u, v \in \supp(z), \, u \neq \pm v$ with $ u\pm v \in 2L^* $ and $ z(v) \le z(u) $, then $ (C,\tilde{z}) $ with
		\begin{align*}
			\tilde{z}( t) = \begin{cases}
				z(u) - z(v), &\text{if } t=  \pm u  \\
				0, &\text{if } t = \pm v,  \\
				2z(v)+ z(t), &\text{if } t  \in \{\frac{\pm u \pm v}{2} \},    \\
				z(t), &\text{otherwise,}
			\end{cases}
		\end{align*}
		is a solution for~\eqref{eq:primal-simple}.
		\item  If there is $ u \in \supp(z) $ and $ u = kv  $ for some integer $ k \ge 2 $,
		then $ (C,\tilde{z}) $ with
		\begin{align*}
			\tilde{z}(t) = \begin{cases}
				 0,&\text{if } t = \pm u,  \\
				 z(v) + k^2z(u), &\text{if } t =\pm
				  v,  \\
				z(t), &\text{otherwise,}
			\end{cases}
		\end{align*}
		is a solution for~\eqref{eq:primal-simple}.
	\end{enumerate}
	In both cases, $ \tilde{z} $ satisfies
	\begin{align*}
		\sum_{t \in L^*} z(t) tt^\mathsf{\top} = \sum_{t \in L^*} \tilde{z}(t) tt^\mathsf{\top} \quad \text{and} \quad \sum_{t \in L^*} z(t) < \sum_{t \in L^*} \tilde{z}(t).
	\end{align*}
\end{lemma}
\begin{proof}
	\textit{(1)}
	By construction, we have
	\[
	 \sum_{t \in L^*} z(t) < \sum_{t \in L^*} z(t)+ 4z(v) =  \sum_{t \in L^*} \tilde{z}(t) .
	\]
	Computing
	\begin{align*}
		& z(u)uu^{\sf{\top}} + z(v) vv^{\sf{\top}} =  (z(u)-z(v)) uu^{\sf{\top}} + z(v) (uu^\mathsf{\top}+ vv^\mathsf{\top})   \\
		= \, &(z(u)-z(v)) uu^{\sf{\top}} + 2z(v) \left (\left(\frac{u+v}{2} \right) \left(\frac{u+v}{2}\right)^\mathsf{\top}  +\left(\frac{u-v}{2} \right) \left(\frac{u-v}{2}\right)^\mathsf{\top} \right),
	\end{align*}
	(and analogously for the pair $ -u,-v $) we obtain $ \sum_{t \in L^*} z(t) tt^\mathsf{\top} = \sum_{t \in L^*} \tilde{z}(t) tt^\mathsf{\top} $ and $ CI-4\pi^2 \sum_{t \in L^*} \tilde{z}(t) tt^\mathsf{\top}  \in \mathcal{S}_n^+  $.

	Moreover, by the subquadratic inequality,
	\begin{align*}
		&1-\cos(2\pi u^\mathsf{\top}x) + 1-\cos(2\pi v^\mathsf{\top}x)   \\
		=\, & 1-\cos\left(2\pi  \left(\frac{u+v}{2} + \frac{u-v}{2}\right)^\mathsf{\top}   x\right) + 1-\cos\left(2\pi \left(\frac{u+v}{2} - \frac{u-v}{2}\right)^\mathsf{\top}x\right)  \\
		\le\, & 2 \left(1-\cos\left(2\pi \left(\frac{u+v}{2}\right)^\mathsf{\top}x\right)\right) + 2\left(1-\cos\left(2\pi \left(\frac{u-v}{2}\right)^\mathsf{\top}x\right)\right).
	\end{align*}
	Thus, for every $ x \in V(L) $
	\begin{align*}
		|x|^2 \leq 2 \sum_{t \in L^*} z(t) (1- \cos(2\pi t^{\sf T} x)) \le  2 \sum_{t \in L^*} \tilde{z}(t) (1- \cos(2\pi t^{\sf T} x)),
	\end{align*}
	implying that $ (C,\tilde{z}) $ is feasible for~\eqref{eq:primal-simple} with the desired properties.

	\textit{(2)} The proof is analogous to (1).
\end{proof}

The lemma gives rise to a procedure that transforms a feasible
solution $ (C,z) $ into a solution $ (C,\tilde{z}) $ such that
$ \tilde{z} $ has only support on at most one primitive lattice
element per coset $ u+2L^* $. Roughly speaking, start with any
solution and apply the above lemma ``as long as possible'', i.e.\ as
long as there are pairs of vectors that satisfy \textit{(1)} or
\textit{(2)} of the above lemma. We do not know how to turn this procedure into an algorithm that terminates after finitely many steps,
but we can show that the procedure converges to a solution with the desired properties.

\begin{theorem}\label{thm:finite-support}
	For any $ n $-dimensional lattice $ L $, the torus $ \R^n /L $ has a finite-dimensional least Euclidean distortion embedding.
	In particular, the program~\eqref{eq:primal-simple} has an optimal solution $ (C,z) $ such that
	\begin{enumerate}
		\item \label{item:finite-supp}  $| \supp(z)  \cap (v+2L^*) | \le 1$ for every coset $ v+2L^* $ of $  L^*/2L^* $.
		\item \label{item:primitive} Every $ u \in \supp(z) $ is primitive in $ L^* $.
	\end{enumerate}
\end{theorem}

Note that claim \textit{(1)} shows that there are at most $2^n -1$ non-zero elements in the support of $z$, therefore we obtain an embedding into a space of dimension at most $2^n-1$.

\begin{proof}
	Let 
	$ (C,z_0) $ with $ z_0 \in \ell^1(L^*) $ be an optimal solution for~\eqref{eq:primal-simple}.
	Our goal is to construct a sequence of solutions $ (C,z_m) $ for~\eqref{eq:primal-simple} that converges to a solution that satisfies \textit{(1)} and \textit{(2)}.
	Let
	\begin{align}\label{eq:A_z-B_z}
		A_z &=  \{  \{u,v \} \,   : \, u \neq \pm v, u \pm v \in 2L^*, u,v \in \supp(z)  \}
	\end{align}
	and let $ (z_m)_m $ be a sequence where $ z_m $ is obtained from $ z_{m-1} $ by applying transformation \textit{(1)} of Lemma~\ref{lemma:support-decrease-simple} to an arbitrary pair $ \{u,v\} \in A_z $ so that $Z_m = \sum_{u \in L^*} z_m(u) $ is maximized (which exists by weak-$*$ compactness). 
	Due to Lemma~\ref{lemma:support-decrease-simple}, the pair $ (C,z_m) $ is feasible for~\eqref{eq:primal-simple} and we have
	\begin{align*}
		\sum_{u \in L^*} z_m(u) uu^\mathsf{\top} = \sum_{u \in L^*} z_{m+1}(u) uu^\mathsf{\top} \quad \text{and} \quad
		Z_m  < Z_{m+1} \quad \text{for all } m \in \N.
	\end{align*}
	The sequence $ Z_m $ is monotonously increasing but bounded, hence, by monotone convergence, the sequence $ Z_m $ converges.

	Now we claim that $ \lim_{m \to \infty} z_m(u) $ exists for all $ u \in L^* $.
	Therefore, assume that $ \{u_m,v_m\} \in A_{z_{m-1}} $ is chosen in the iteration from $ z_{m-1} $ to $ z_m $.
	Assume that $ z_{m-1}(u_m) \ge  z_{m-1}(v_m)$.
	Then, using Lemma~\ref{lemma:support-decrease-simple}, we obtain
	\begin{align*}
		\sum_{u \in L^*}| z_{m}(u) - z_{m-1}(u) | =3\cdot 4z_{m-1}(v_m) = 3(Z_m - Z_{m-1}).
	\end{align*}
	The right hand side converges to zero, therefore the sequence $ z_m $ converges pointwise, i.e.\ there is $ z \in \ell^1(L^*) $ such that
	\begin{align*}
		\lim_{m \to \infty} z_m(u) = z(u)  \quad \text{for all } u \in L^*.
	\end{align*}
	Next, we will show that $ z $ satisfies \textit{(1)}, which is equivalent to $ A_z = \emptyset $. But this simply follows by construction:
	For every pair $ \{u,v\} \in A_{z_m}  $ we have
	\begin{align*}
		\lim_{m \to \infty} \min \{ z_m(u), z_m(v) \} = 0.
	\end{align*}
	This holds because if there was $ \{u,v\} \in A_z  $ and $ \varepsilon > 0 $ such that for all $ M $ there was $ m \ge M$
	with $ \min \{ z_m(u), z_m(v) \} \ge \varepsilon$, then according to the construction of $ z_m $ there would also be $ m' \ge m$ with
	\begin{align*}
		Z_{m'} -Z_{m} \ge  4\min \{ z_m(u), z_m(v) \}\ge 4\varepsilon .
	\end{align*}
	This would be a contradiction to the convergence of the sequence $ Z_m $.

	Now, if there is $ u \in \supp(z) $ with $ u = kv $ for some $ k \ge 2 $, we may apply \textit{(2)} of Lemma~\ref{lemma:support-decrease-simple} to obtain a new feasible solution $ \tilde{z} $ with $ v \in \supp(\tilde{z}) $ and $ z(u) = 0 $. This solution may contain a pair $ (u,v) \in A_{\tilde{z}} $. But in this case, we may again apply \textit{(1)} of Lemma~\ref{lemma:support-decrease-simple}.

	By continuing like this, we will finally end up with a solution that has only support on primitive vectors and on one vector per coset, thus satisfying properties \textit{(1)} and \textit{(2)}.
\end{proof}

Unfortunately, the proof does not give a bound on
$ \max \{|u| \, : \, u \in \supp(\tilde{z})\} $ for $ \tilde{z} $ constructed in Theorem~\ref{thm:finite-support}.

\section{Improved lower bound}
\label{sec:lower-bound}

In this section we apply Theorem~\ref{thm:dual-simple} to get a
constant factor improvement over~\eqref{eq:HRbound}, basically without
any effort.

\begin{theorem}
  \label{thm:lowerbound1}
Let $L$ be an $n$-dimensional lattice, then
\[
    c_2(\R^n/ L) \geq \frac{\pi \lambda(L^*) \mu(L)}{\sqrt{n}}.
\]
\end{theorem}

\begin{proof}
  Let $y$ be a vertex of the Voronoi cell $V(L)$ which realizes the
  covering radius, that is $|y| = \mu(L)$ and so $y$ is a ``deep hole'' of
  $L$. Choose $\nu = \frac{\lambda(L^*)^2}{2n} \delta_y$ to be a point
  measure supported at $y$ and set $Y = \frac{1}{n} I$. Then
  $(\nu, Y)$ is feasible for~\eqref{eq:dual} because
  \[
  \int_{V(L)} (1 - \cos(2\pi u^{\sf T} x)) \, d\nu(x) =
  (1-\cos(2\pi  u^{\sf T} y)) \frac{\lambda(L^*)^2}{2n}
  \leq \frac{\lambda(L^*)^2}{n} \leq \frac{|u|^2}{n} = u^{\sf T} Y u
 \]
 for every $u \in L^* \setminus \{0\}$.  Hence, by
 Theorem~\ref{thm:dual-simple},
\[
  c_2(\R^n / L)^2 \geq 2 \pi^2 \int_{V(L)} |x|^2 \, d\nu(x) =
  \frac{\pi^2 \lambda(L^*)^2 \mu(L)^2}{n}. \qedhere
\]
 \end{proof}

 \section{Least distortion embeddings of $\R^n / \Z^n$ and of orthogonal decompositions}
 \label{sec:standard-torus}

 As our second application of Theorem~\ref{thm:dual-simple}, through
 Theorem~\ref{thm:lowerbound1}, we prove that the standard
 embedding~\eqref{eq:standard-embedding} of the standard torus is
 indeed a least distortion embedding. It is somewhat surprising that
 this result is new. We also note that one can easily use the same
 argument to capture the case of flat tori whose lattices have an
 orthogonal basis.

 \begin{theorem}\label{thm:standard-torus-embedding}
   The standard embedding $\varphi : \R^n / \Z^n \to \R^{2n}$ of the
   standard torus $\R^n / \Z^n$ given by
   \[
  \varphi(x_1, \ldots, x_n) = (\cos 2\pi x_1, \sin 2\pi x_1, \ldots,
  \cos 2\pi x_n, \sin 2\pi x_n)
\]
is a least distortion embedding with distortion $c_2(\R^n/\Z^n) = \pi/2$.
\end{theorem}

\begin{proof}
  We have $\lambda(\Z^n) = 1$ and $\mu(\Z^n) = \sqrt{n/4}$, so
  $c_2(\R^n / \Z^n) \geq \pi/2$ by Theorem~\ref{thm:lowerbound1}.

  To show the corresponding upper bound we show that the embedding
  \[
    \phi(x_1, \ldots, x_n) =\frac{1}{\sqrt{32}}  ( e^{-2\pi
      ix_1},\ldots, e^{-2\pi ix_n})
  \]
  has contraction $1$ and expansion $\pi/2$. Then turning $\phi$ into a
  real embedding and rescaling does not change the distortion and
  gives the standard embedding $\varphi$.

  To show that $\phi$ has contraction $1$ and expansion $\pi/2$ it
  suffices to prove that $(\frac{\pi^2}{4},z)$ with
\[
   z(u) = \begin{cases}
    \frac{1}{32} & \text{if } u= \pm e_i,\\
    0 & \text{otherwise},
    \end{cases}
\]
is a feasible solution for~\eqref{eq:primal-simple}.

The expansion equals $\pi/2$ because
\[
  C I - 4 \pi^2 \sum_{u \in \Z^n} z(u) uu^{\sf T} = \frac{\pi^2}{4} I -
  4 \pi^2 \frac{2}{32} I = 0.
\]
Moreover, to show that the contraction equals $1$ we need to verify
the inequality
\begin{equation}
\label{eq:heart-inequality}
    \sum_{i=1}^n x_i^2 \leq \frac{4}{32} \sum_{i=1}^n (1-\cos(2\pi x_i)) \quad \text{for all } x \in V(\Z^n) = [-1/2,1/2]^n,
\end{equation}
which we check summand by summand, that is
$x_i^2 \leq \frac{1}{8} (1 - \cos(2\pi x_i)) = \frac{1}{4} \sin(\pi x_i)^2$, where we have equality
if $x_i = \pm 1/2$. This follows directly from the standard inequality $\sin \alpha \geq \frac{2}{\pi} \alpha$, for $\alpha \in [0,\pi/2]$. 
\end{proof}

Here it is interesting to note that even for the rather trivial
standard embedding of the standard torus the structure of the most
contracted pairs is rich. Every center of every face of the Voronoi
cell $V(\Z^n)$ gives a most contracted pair, see Figure~\ref{fig:lat:90}.



\medskip

Recapitulating the above proof, one recognizes that at its heart is the
verification of inequality~\eqref{eq:heart-inequality}. Here one
reduces the situation from $\Z^n$ to $\Z$. This works because $\Z^n$
can be orthogonally decomposed as the direct sum of $n$ copies of $\Z$
and this can be done in generality as the following theorem
demonstrates.

\begin{theorem}
\label{thm:direct-sum}
Let $L \subseteq \R^n$ be a lattice such that $L$ decomposes as the
orthogonal direct sum of lattices $L_1, \ldots, L_m$, i.e.
\begin{equation*}
  L = L_1 \perp L_2 \perp \ldots \perp L_m.
\end{equation*}
Then
\begin{equation*}
c_2(\R^n/L) = \max \{ c_2(\R^{n_j}/L_j) : j = 1, \ldots, m\},
\end{equation*}
where $\R^{n_j}$ is (isometric) to the Euclidean space spanned by
$L_j$.
\end{theorem}

\begin{proof}
  Any Euclidean embedding of $\R^n/L$ gives a Euclidean embedding of
  $\R^{n_j}/L_j$ which immediately gives the inequality
  $c_2(\R^n/L) \geq \max \{ c_2(\R^{n_j}/L_j) : j = 1, \ldots, m\}$.

  Also the reverse inequality is easy to see. Let
  $ \varphi_j : \R^{n_j} / L_j \to H_j$ be a Euclidean embedding of
  $ \R^{n_j}/L_j $ with distortion $C_j$ scaled so that the
  contraction is $1$ and the expansion is $C_j$. We identify
  $ \R^n /L \cong \R^{n_1} /L_1\perp \ldots \perp \R^{n_m}/L_m $,
  write $ x \in \R^n/L $ as $x = (x_1,\ldots,x_m) $ with
  $x_j \in \R^{n_j}/L_j $ so that
  $d_{\R^n/L}(x,y)^2 = \sum_{j = 1}^m d_{\R^{n_j} /L_j}(x_j,y_j)^2
  $. Then
\[
\varphi: \R^n /L \to H:= H_1 \perp \ldots \perp H_m,  \quad   (x_1,\ldots,x_m) + L  \mapsto (\varphi_1(x_1), \ldots, \varphi_m(x_m))
\]
is a Euclidean embedding of $\R^n/L $ into the Hilbert space $H
$. Its distortion is at most $\max\{C_j : j = 1, \ldots, m\}$ because
for every pair $x, y \in \R^n / L$
\[
\begin{split}
  |\varphi(x)- \varphi(y)|^2 & = \sum_{j = 1}^m |\varphi_j(x_j)- \varphi_j(y_j)|^2 \leq \sum_{j = 1}^m C_j^2 d_{\R^{n_j}/L_j}(x_j,y_j)^2   \\
  & \leq \max_{j=1, \ldots, m} C_j^2 \; \sum_{j = 1}^m
  d_{\R^{n_j}/L_j}(x_j,y_j)^2 = \max_{j=1, \ldots, m} C_j^2 \;
  d_{\R^n/L}(x,y)^2,
\end{split}
\]
showing that the expansion of $\varphi$ at most
$\max\{C_j : j = 1, \ldots, m\}$ and, in exactly the same way, one
shows that the contraction of $\varphi$ is at most $1$.
\end{proof}

\section{Least Euclidean distortion embedding of the lattice $ D_n^*
  $}
\label{sec:dnstar}

In this section, we will construct a least Euclidean distortion
embedding for the lattice $ D_n^* $. To construct this embedding we
have to assign weights $z(u)$ for $u \in (D_n^*)^* = D_n $, a root
lattice,  which is
defined as
\begin{align*}
	D_n = \left\{ u \in \Z^n \, : \, \sum_{i = 1}^n u_i \text{ even}  \right\}.
\end{align*}
The shortest vectors of $ D_n $ are precisely its \emph{roots}, that is,
the vectors of squared length $ 2 $. They are given by
$ R(D_n) = \{\pm (e_i \pm e_j) :  i \neq j \}$.

To prove optimality of the embedding for $ D_n^* $ we will make use of
the lattice' symmetries. Generally, these symmetries
can be exploited to construct a symmetrized solution
for~\eqref{eq:primal-simple} as explained below.

Let $ G_{L^*} $ be the orthogonal group of the lattice $ L^*$, which
is the group of orthogonal matrices $ S $ satisfying $ SL^* = L^*
$. Note, that by definition of $ L^* $, the group $ G_L $ equals $G_{L^*}$.

Now if $ (C,z) $ is a solution of~\eqref{eq:primal-simple}, then this
solution can be symmetrized to a solution $ (C,\bar{z}) $ by taking
the group average over $ G_L $:
\begin{align*}
	\bar{z}(u) = \frac{1}{|G_L|} \sum_{S \in G_L} z(Su). 
\end{align*}

The tuple $ (C,\bar{z}) $ is feasible because $ \bar{z} \ge 0 $ and 
\begin{align*}
	CI - 4\pi^2\sum_{u \in L^*} \bar{z}(u) uu^{\sf{\top}} = \frac{1}{|G_L|} \sum_{S \in G_L} S^{\sf{\top}} (CI - 4\pi^2\sum_{u \in L^*} z(u)  uu^{\sf{\top}} )  S \in \mathcal{S}^n_+. 
\end{align*}
Moreover, for all $ x \in V(L) $: 
\begin{align*}
	2\sum_{u \in L^*} \bar{z}(u) (1-\cos(2\pi u^{\sf{\top}}x)) &= \frac{2}{|G_L|}\sum_{S \in G_L} \sum_{u \in L^*} z(Su) (1-\cos(2\pi u^{\sf{\top}}x))  \\
	&= \frac{2}{|G_L|} \sum_{S \in G_L} \sum_{u \in L^*} z(u) (1-\cos(2\pi (S^{-1}u)^{\sf{\top}}x))    \\
	&= \frac{2}{|G_L|} \sum_{S \in G_L} \sum_{u \in L^*} z(u) (1-\cos(2\pi u^{\sf{\top}}(Sx)))  \\
	&\ge \frac{1}{|G_L|} \sum_{S \in G_L}  |Sx|^2 \\
	&= |x|^2.
\end{align*}
By construction, $ \bar{z} $ is constant on orbits under the action of $ G_L $:
\begin{align}\label{eq:symmetrized-coefficients}
	\text{If } Su = v \text{ for some S} \in G_L, \text{ then } \bar{z}(u) = \bar{z}(v)
\end{align}
Moreover, we have 
\begin{align}\label{eq:no-coefficient-increase}
	\sum_{u \in L^*} z(u) = \sum_{u \in L^*} \bar{z}(u)
\end{align}

Back to $D_n^*$. The matrix group generated by the reflections
$ (I-xx^{\sf T}) $ for $ x \in \pm (e_i \pm e_j) $, for $ i \neq j $,
is called the \emph{Weyl group} of $ D_n $ and preserves the
lattice. It acts transitively on the roots; see, for examples,
\cite[Chapter~1]{Ebeling1994a} for details.

\begin{theorem}
An optimal solution for~\eqref{eq:primal-simple} for $D_n^*$ is given by $ (C,z) $ where 
	\begin{align*}
		z(u) = \begin{cases}
			\alpha, &\text{ if } u \in R(D_n) ,   \\
			0 &\text{otherwise},
		\end{cases}
	\end{align*}
	and 
	\begin{align*}
		C  = \max\left\{
		\frac{|x|^2}{2\sum_{u \in R(D_n)} (1- \cos(2\pi
			u^{\mathsf{T}}x))  } : x \in V(L)\setminus \{0\} \right\},
	\end{align*}
	and $ \alpha = \frac{C}{16 \pi^2 (n-1)}$.
      \end{theorem}
      
\begin{proof}
  Let $ (C,z) $ be an optimal solution of~\eqref{eq:primal-simple}. By
  Theorem~\ref{thm:finite-support}, we may assume that every element
  in $ \supp(z) $ is primitive and there is maximally one pair
  $ \pm u \in \supp(z) \cap v+2D_n$ for every coset $ v+2D_n $. Now
  suppose that there is $ u \in \supp(z) $ such that $ |u|^2 > 2 $.
  Assume that $ u = \sum_{i = 1}^n u_i e_i \in \supp(z)$ with
  $ u_i \in \Z $ and $ \sum_{i = 1}^n u_i $ even. Furthermore assume
  without loss of generality that $ u_1+u_2 $ is even (such a pair
  exists, as $ u \in D_n $).
	
                Now define
	\begin{align*}
		v = (I-(e_1+e_2)(e_1+e_2)^{\sf T}) u = -u_2e_1-u_1e_2 + \sum_{i =3}^n u_i e_i . 
	\end{align*}
	As $ \supp(z) $ contains only primitive vectors and $ |u|^2 > 2 $, it follows that $ v \neq -u $
	 
	Using~\eqref{eq:symmetrized-coefficients}, we can symmetrize
        the solution $ (C,z) $ to a solution $ (C,\bar{z}) $ which
        satisfies $ \bar{z}(u) = \bar{z}(v) $.  Since $ u_1+u_2 $ is
        even, it follows that $ \frac{u \pm v}{2}\in D_n$ and we can
        apply Lemma~\ref{lemma:support-decrease-simple} to obtain a
        new solution $ (C, \tilde{z}) $.
	
	Iterating this process and arguing as in the proof of
        Theorem~\ref{thm:finite-support}, we will obtain a sequence of
        solutions which will finally converge to a solution which has
        only support on the points of the form $ \pm (e_i\pm e_j) $
        for $ i \neq j $.  As the Weyl group acts transitively on the
        roots, it follows, that there is an optimal solution
        $ (C,\hat{z}) $ of~\eqref{eq:primal-simple} which has uniform
        support on the roots.
	
	Finally, we have that the matrix $ \sum_{u \in R(D_n)} uu^{\sf{\top}}$ is
        a positive multiple of the identity because the roots of
        $D_n$ form a spherical $ 2 $-design (see~\cite{Venkov2001a}
        for details). More precisely, one we have $\sum_{u \in R(D_n)} uu^{\sf{\top}} = \frac{I |R(D_n)}{n} I$ and $|R(D_n)| = 4\binom{n}{2} = 2n(n-1)$.
\end{proof}

\begin{remark}
  The argument cannot be straightforwardly carried over to the lattice
  $ A_n^* $. However, we conjecture that there is also an optimal
  solution for~\eqref{eq:primal-simple} which assigns uniform weights
  to the roots of $ A_n $. This has been proved by Moustrou and
  Vallentin~\cite{Moustrou2023a} for $n = 2$.
\end{remark}

\section{Least distortion embeddings of two-dimensional flat tori}
\label{sec:2d-embeddings}

In this section we will construct least Euclidean distortion
embeddings of flat tori in dimension $2$.

First, as a simple corollary of Theorem~\ref{thm:primal-simple}, we
will give a recipe to construct (possibly non-optimal) embeddings of
flat tori of arbitrary dimension provided that they satisfy the
following assumption:
\begin{align}\label{eq:identity-assumption}
	\textup{There exist } u_1,\ldots, u_k \in L^*,  \, z_1,\ldots,z_k \ge 0 \text{ such that }
	4\pi^2 \sum_{i = 1}^k z_i u_i u_i^{\sf T} = I.
\end{align}

As we will prove in Lemma~\ref{lemma:id-decomp-for-2d-case},
condition~\eqref{eq:identity-assumption} can be realized for every
$2$-dimensional lattice. Another example of lattices that satisfy
condition~\eqref{eq:identity-assumption} are duals of root
lattices. If $L^*$ is a root lattice, then
assumption~\eqref{eq:identity-assumption} is satisfied. In this case
the root system $R \subseteq L^*$ of $L^*$ forms a spherical
$2$-design, implying that
\[
4\pi^2 \alpha \sum_{u \in R} uu^{\sf T} =  I
\]
for some positive constant $\alpha$\footnote{The constant being $\alpha = \frac{|R|}{2\pi^2 n}$ as $\sum_{u \in R} uu^{\sf T} = \frac{2|R|}{n} I$.}. We refer to the monograph by
Venkov~\cite{Venkov2001a} for more information on root lattices and
spherical designs.

One referee pointed out that every lattice satisfies condition~\eqref{eq:identity-assumption}. This follows from \cite[Theorem 2.5]{Agostini2019a}: There is a unique discrete Gaussian distribution $p : L^* \to \R$ so that
\[
\sum_{u \in L^*} p(u) uu^{\sf T} =  I
\]
holds. 

\begin{corollary}
 \label{cor:general-embedding}
 Let $L \subseteq \R^n$ be a lattice that satisfies
 \eqref{eq:identity-assumption}.  Then
 \[
\varphi : \R^n/L \to \C^{k} ,  \quad
\varphi(x) = (\sqrt{D z_{1}} e^{2i\pi u_{1}^{\sf T}x}, \ldots, \sqrt{D z_{k}} e^{2i\pi
u_{k}^{\sf T}x})
\]
with
\begin{equation}
\label{eq:max-contraction-coeff}
D = \max\left\{
\frac{|x|^2}{2 \sum_{i = 1 }^k z_{i} (1- \cos(2\pi
  u_{i}^{\sf T}x))  } : x \in V(L)\setminus \{0\} \right\}
\end{equation}
is a Euclidean embedding of $\R^n /L$ with distortion
$\sqrt{D}$. In particular,
\begin{equation}
 \label{eq:lower-bound-from-general-construction}
c_2(\R^n /L)^2 \le D.
\end{equation}
\end{corollary}

\begin{proof}
  The pair $((Dz_{i})_{1 \le i \le k}, D)$ is a feasible solution for
  the primal optimization problem \eqref{eq:primal-simple}.
\end{proof}

Except for the easiest case of the standard torus we do not know how
to determine $D$ explicitly.  Unfortunately, it seems to be difficult
to compute most contracted pairs $(0,x$), i.e.\ vectors $x \in V(L)$
that are maximizers of the right hand side
of~\eqref{eq:max-contraction-coeff}.

\medskip

Next, we show that Corollary~\ref{cor:general-embedding} can be applied
to every $2$-dimensional lattice. For this we will use the concept of
an \emph{obtuse superbasis}. An obtuse superbasis of an
$n$-dimensional lattice $L$ is a basis $u_1, \ldots, u_n$ of $L$
enlarged by the vector $u_0 = -u_1 - \cdots - u_n$ so that these
$n + 1$ vectors pairwise form non-acute angles, i.e.\
\begin{equation}
 \label{eq:obtuse-superbase-cond}
 u_i^\mathsf{T} u_j \le 0 \; \text{ for all } \; 0 \leq i < j \leq
 n.
\end{equation}
It is known that up to dimension $3$ all lattices have an obtuse
superbasis, but from dimension $4$ on this is no longer the case, see
for instance~\cite{Conway1992}.

\begin{lemma}
  \label{lemma:id-decomp-for-2d-case}
  If $L$ is a two-dimensional lattice, then its dual lattice
  $L^*$ satisfies~\eqref{eq:identity-assumption}. More precisely, one can choose any obtuse superbasis $u_0,u_1,u_2$ of $L^*$ for
  the representation~\eqref{eq:identity-assumption}.
\end{lemma}

\begin{proof}
  Let $u_0,u_1,u_2$ be an obtuse superbasis of $L^*$.  We will show
  that there are non-negative coefficients $z_0, z_1, z_2$ such that
 \[
   I = z_0 u_0u_0^\mathsf{T} + z_1 u_1u_1^\mathsf{T} + z_2
   u_2u_2^\mathsf{T} ,
\]
and therefore condition~\eqref{eq:identity-assumption} holds.

We may assume that $|u_1| \geq |u_0| = 1$, by
scaling and renumbering. Then, by Gram-Schmidt orthogonalization,
$u_0$ and $w_1 := u_1 - (u_0^\mathsf{T} u_1) u_0$ are orthogonal and
so
\[
\begin{split}
  I & = u_0 u_0^\mathsf{T} + \frac{1}{|w_1|^2}w_1w_1^\mathsf{T}\\
  & = \left(1+\frac{ (u_0^\mathsf{T} u_1)^2}{|w_1|^2}\right) u_0u_0^\mathsf{T} +
  \frac{1}{|w_1|^2} u_1u_1^\mathsf{T} - \frac{u_0^\mathsf{T}
    u_1}{|w_1|^2}(u_0u_1^\mathsf{T}+u_1u_0^\mathsf{T}).
\end{split}
\]
Using
\[
  u_2 u_2^{\sf T} = (-u_0 - u_1) (-u_0 - u_1)^{\sf T} = u_0 u_0 + u_1 u_0^{\sf T} + u_0
  u_1^{\sf T} + u_1 u_1^{\sf T},
\]
yields
\[
I =\left(1+\frac{ (u_0^\mathsf{T} u_1)^2 +
    u_0^\mathsf{T}u_1}{|w_1|^2}\right)u_0u_0^\mathsf{T} +\frac{1+
  u_0^\mathsf{T}u_1}{|w_1|^2}u_1u_1^\mathsf{T} - \frac{ u_0^\mathsf{T} u_1}{|w_1|^2}u_2u_2^\mathsf{T}.
\]
To prove that the three coefficients in the above sum are
non-negative, observe for the third coefficient that $u_0^\mathsf{T}u_1 \le 0
$. For the second coefficient
\[
0 \le  - u_0^\mathsf{T}u_2  = u_0^\mathsf{T}u_0 + u_0^\mathsf{T} u_1 =
1 +u_0^\mathsf{T} u_1.
\]
For the first coefficient we compute the squared norm $|w_1|^2 =
|u_1|^2 - (u_0^{\sf T} u_1)^2$ and see that the first coefficient is
nonnegative if and only if $|u_1|^2 \geq -u_0^{\sf T} u_1$, which is
true because $|u_1| \geq |u_0| = 1$ by assumption.
\end{proof}

Now, to verify that the embedding of
Corollary~\ref{cor:general-embedding} is indeed a least Euclidean
distortion embedding for two-dimensional flat tori, we construct a
dual solution for \eqref{eq:dual} that shows that the
upper bound~\eqref{eq:lower-bound-from-general-construction} is sharp.

\begin{theorem}
\label{thm:dual-solution-for-2d-case}
Let $L \subseteq \R^2$ be a $2$-dimensional lattice, then
$c_2(\R^2/L)^2 = D$, where $D$ is defined
in~\eqref{eq:max-contraction-coeff}.
\end{theorem}

\begin{proof}
  Let $u_0,u_1,u_2$ be an obtuse superbasis of $L^*$. We may assume,
  see for example \cite{Conway1992}, that this superbasis is chosen in
  a way such that $u_i$ is a shortest vector in its coset
  $u_i + 2L^*$, with $i = 0, 1, 2$. By
  Lemma~\ref{lemma:id-decomp-for-2d-case} we can determine coefficients
  $z_0,z_1,z_2 \ge 0$ such that
  $4\pi^2 \sum_{i = 0}^2 z_i u_iu_i^\mathsf{T} = I$.

  Furthermore, let $\bar{x} \in V(L)$ be a vector such that $(0,\bar{x})$
  is a most contracted pair for the embedding~$\varphi$ of
  Corollary~\ref{cor:general-embedding}, that is, $\bar{x}$ is a
  maximizer for~\eqref{eq:max-contraction-coeff}.

  We define the pair $(Y,\nu)$ as follows: Set $\beta =
  \frac{D}{2\pi^2 |\bar{x}|^2}$, define $Y$ via
    \begin{align}
      \label{eq:linear-eqs-for-Y-in-2d-case}
      \Tr(u_{i}u_{i}^\mathsf{T}Y) =  \beta (1-\cos(2\pi u_{i}^\mathsf{T} \bar{x})), \quad i \in \{0,1,2\},
    \end{align}
    and let $\nu = \beta \delta_{\bar{x}}$ be a point measure supported
    only on $\bar{x}$. We now verify that this pair is a feasible dual
    solution with objective value $D$.

    We have
	\begin{align*}
		\Tr(Y) = \Tr(Y I) = 4\pi^2 \sum_{i} z_i
          \Tr(u_iu_i^\mathsf{T} Y) = 4 \pi^2 \beta  \sum_i z_i (1-\cos(2\pi u_{i}^\mathsf{T}\bar{x} )) = 1.
	\end{align*}

	Equation~\eqref{eq:linear-eqs-for-Y-in-2d-case} together with
        Lemma~\ref{lem:equality-for-Vor-rel-vecs} implies
	\begin{align*}
		\Tr(u u^\mathsf{T}Y)\geq  \beta (1-\cos(2\pi
          u^\mathsf{T} \bar{x})) \quad \text{ for all } u \in L^*.
	\end{align*}

	Finally, it remains to show that $Y$ is positive semidefinite.
        For this we compute its Gram matrix $B$ with respect to $u_0,
        u_1$, that is
	\begin{align*}
		B_{ij} =u_i^\mathsf{T}Y u_j = \frac{1}{2}\Tr((u_i u_j^\mathsf{T} + u_j u_i^\mathsf{T})Y), \quad 0 \le i,j \le 1.
	\end{align*}
	Then $B_{ii} = \Tr(u_iu_i^\mathsf{T}{Y}) =
        2 \beta \sin^2(\pi u_i^\mathsf{T}\bar{x})$.
	Since
	\begin{align*}
		u_0 u_1^\mathsf{T} + u_1 u_0^\mathsf{T} = u_{2}u_{2}^\mathsf{T} - u_0u_0^\mathsf{T} - u_1u_1^\mathsf{T},
	\end{align*}
	we get
	\begin{align*}
		B_{01} &= \frac{\beta}{2}\Big (2\sin^2(\pi u_{2}^\mathsf{T}\bar{x}) - 2\sin^2(\pi u_{1}^\mathsf{T}\bar{x}) - 2\sin^2(\pi u_{0}^\mathsf{T}\bar{x}) \Big )   \\
		&=2\beta \sin(\pi u_{0}^\mathsf{T}\bar{x})\sin(\pi u_{1}^\mathsf{T}\bar{x})\cos(\pi (u_0+u_1)^\mathsf{T}\bar{x}).
	\end{align*}
        From this we see that matrix $B$ is the Schur-Hadamard
        (entry-wise) product of the positive semidefinite rank-one
        matrix $xx^\mathsf{T}$ with
        $x_i = \sqrt{2\beta} \sin(\pi u_i^\mathsf{T}\bar{x})$ and the
        symmetric matrix $M \in \R^{2 \times 2}$ defined by
	\begin{align*}
		M_{ij} = \begin{cases}
			1 &\text{ if } i = j  \\
			\cos(\pi (u_0+u_1)^\mathsf{T}\bar{x}) &\text{ if } (i,j) \in \{(0,1),(1,0)\}.
		\end{cases}
	\end{align*}
	The matrix $M$ is positive semidefinite because $M_{ii} \ge 0$ and
	\begin{align*}
		\det(M) = 1 - \cos^2(\pi (u_0+u_1)^\mathsf{T}\bar{x}) \ge 0.
	\end{align*}
        Thus $B$, and therefore also $Y$, is positive semidefinite,
        which finishes the proof.
\end{proof}

To conclude the discussion of $2$-dimensional lattices
Figure~\ref{fig:lats:comparison} collects an illustration of the
behavior of the distortion function defined in
\eqref{eq:max-contraction-coeff} and the most contracted pairs,
applying the above results, for a selection of $2$-dimensional
lattices.

\begin{figure}
\centering
\begin{subfigure}{.49\textwidth}
  \centering
  \includegraphics[width=\linewidth]{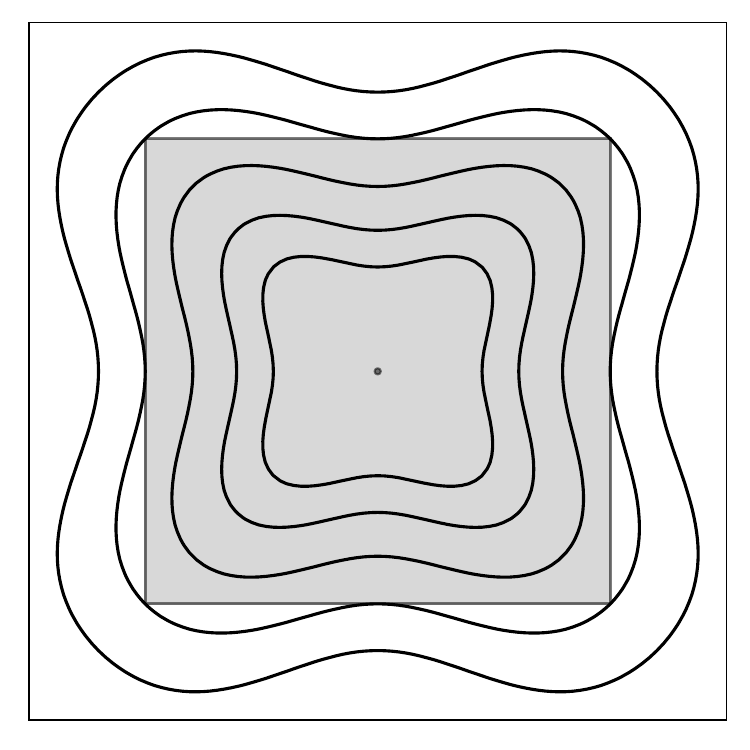}
  \caption{$L_{90^\circ}$ ($\Z^2$)}
  \label{fig:lat:90}
\end{subfigure}%
\hfill
\begin{subfigure}{.49\textwidth}
  \centering
  \includegraphics[width=\linewidth]{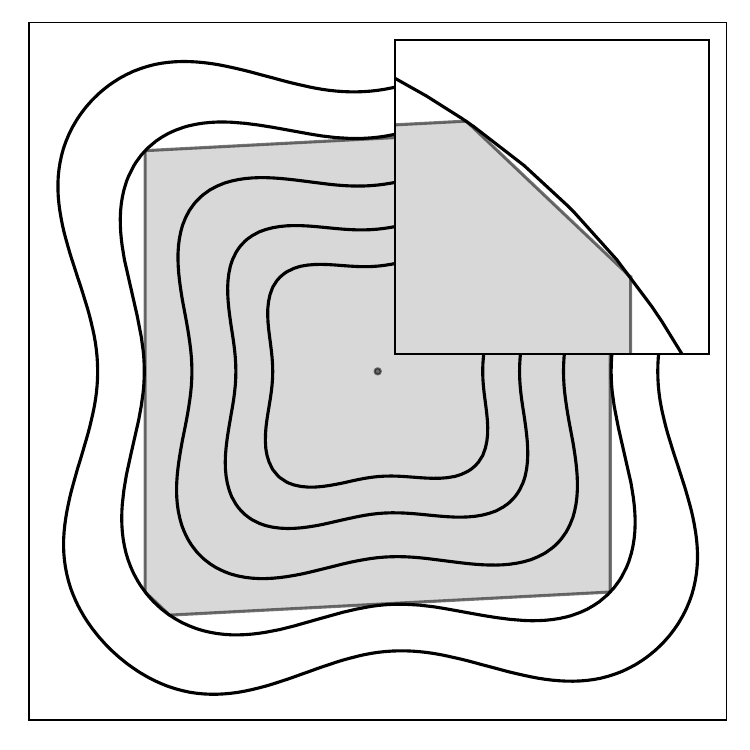}
  \caption{$L_{93^\circ}$}
  \label{fig:lat:93:zoom}
\end{subfigure}
\vskip\baselineskip
\begin{subfigure}{.49\textwidth}
  \centering
  \includegraphics[width=\linewidth]{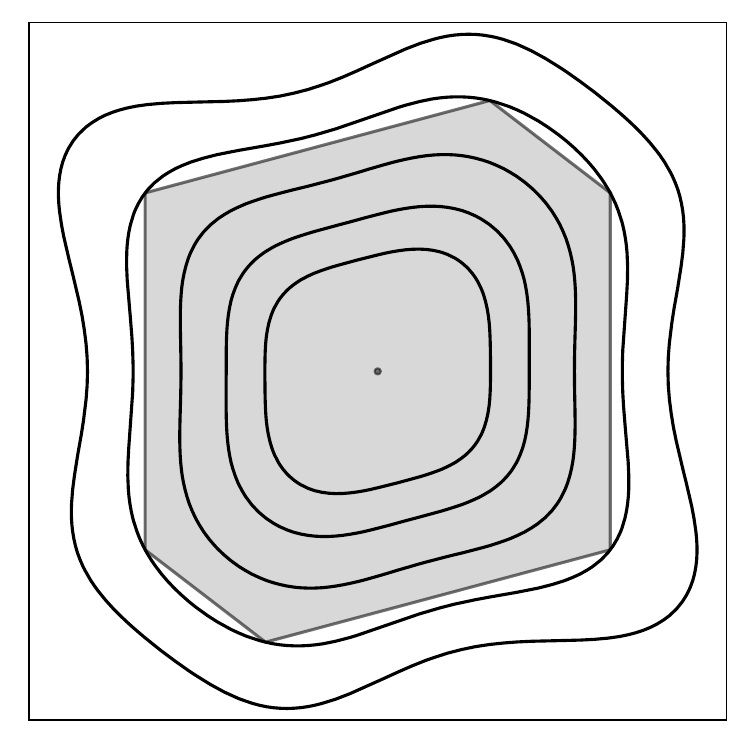}
  \caption{$L_{105^\circ}$}
  \label{fig:lat:105}
\end{subfigure}
\hfill
\begin{subfigure}{.49\textwidth}
  \centering
  \includegraphics[width=\linewidth]{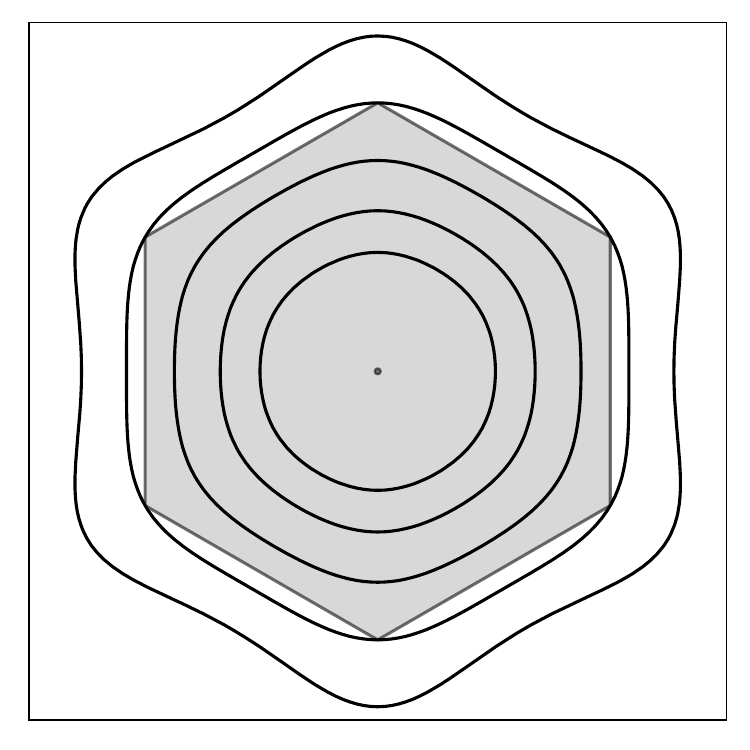}
  \caption{$L_{120^\circ}$ ($A_2^*$)}
  \label{fig:lat:120}
\end{subfigure}%
\caption{
  Let $L_\varphi$ be the lattice spanned by  $v_1 = e_1$, $v_2 = R_\varphi e_1$, where $R_\varphi$ is the rotation by $\varphi$ degrees (counter-clockwise).\\
  The Voronoi cell of any lattice in $\R^2$ is either a rectangle or a hexagon. We plot contour lines of the distortion function for a selection of lattices that illustrate how the distortion function and the most contracted pairs vary with the shape of the Voronoi cell. $L_{93^\circ}$ shows what happens for almost degenerated hexagons, i.e. lattices close to the standard lattice. A zoom into the behavior around the short edge of the hexagon is included to illustrate that the most contracted points are all vertices of the Voronoi-cell.
  }
  \label{fig:lats:comparison}

\end{figure}

\section{Discussion and open questions}
\label{sec:open-questions}

In this paper we derived an infinite-dimensional semidefinite program
to determine the least distortion Euclidean embedding of a flat
torus. It would be very interesting to show that this
infinite-dimensional semidefinite program can in fact be turned into a
finite-dimensional semidefinite program. Then one could, similarly to the case of
finite metric spaces, algorithmically determine least distortion
Euclidean embeddings of flat tori; at least up to any desired
precision.

For this a characterization of the most contracted pairs is needed. We
believe that the most contracted pairs are always of the form $(0,y)$
and $y$ is a center of a face of the Voronoi cell. However, we do not
know whether such a $y$ can only lie on the Voronoi cell's
boundary. We do not even know whether there are only finitely many
most contracted pairs.

Moustrou and Vallentin~\cite{Moustrou2023a} computed the most
contracted pairs for the lattices $A_2^*$ and $E_8$ using a
modification of the linear programming method for spherical designs.

We also do not know how to restrict the variable $z \in \ell^1(L^*)$
to finite dimension, even though Theorem~\ref{thm:finite-support}
shows that we can always find a finite-dimensional least distortion
embedding. Obtaining a bound on the maximally needed length of a
support vector in the cosets $L^*/2L^*$ would solve this problem.

Another interesting problem is to determine $n$-dimensional lattices
which maximize the distortion among all $n$-dimensional lattices.

\section*{Acknowledgements}

We would like to sincerely thank the two anonymous referees, whose thorough reading of the manuscript and many valuable comments and suggestions have greatly helped to improve the paper.

\smallskip

This project has received funding from the European Union's Horizon
2020 research and innovation programme under the Marie
Sk\l{}odowska-Curie agreement No 764759. F.V. is partially supported
by the SFB/TRR 191 ``Symplectic Structures in Geometry, Algebra and
Dynamics'', F.V. and M.C.Z. are partially supported ``Spectral bounds
in extremal discrete geometry'' (project number 414898050), both
funded by the DFG. A.H. is partially funded by the Deutsche Forschungsgemeinschaft (DFG, German Research Foundation) under Germany's Excellence Strategy -- Cluster of Excellence Matter and Light for Quantum Computing (ML4Q) EXC 2004/1 -- 390534769.

\section*{Conflict of interest}

The authors declare that there are no conflicts of interest regarding the publication of this paper.


\begin{thebibliography}{16}

\bibitem{Agarwal2020a}
  I. Agarwal, O. Regev, Y. Tang,
  Nearly optimal embeddings of flat tori,
  APPROX/RANDOM 2020.

\bibitem{Agostini2019a}
D. Agostini, C. Am\'endola,
Discrete Gaussian distributions via theta functions,
SIAM J. Appl. Algebra Geom. 3 (2019), 1--30.

\bibitem{Aharoni1985a}
I. Aharoni, B. Maurey, B.S. Mityagin, Uniform embeddings of metric
spaces and of Banach spaces into Hilbert spaces, Israel J. Math. 52
(1985), 251--265.

\bibitem{BachocGSV2012}
C.~Bachoc, D.C.~Gijswijt, A.~Schrijver, F.~Vallentin, Invariant
semidefinite programs, pp.~219--269 in: Handbook on semidefinite, conic, and
polynomial optimization (M.F.~Anjos, J.B.~Lasserre, eds.),
Springer, 2012.

\bibitem{Barvinok2002}
A. Barvinok,
A course in convexity,
AMS, 2002.

\bibitem{Borrelli2012a}
  V. Borrelli, S. Jabrane, F. Lazarus, B. Thibert,
  Flat tori in three-dimensional space and convex integration,
  Proceedings of the National Academy of Sciences (PNAS)
  109 (2012), 7218--7223.

\bibitem{Butler1972a}
  G.J. Butler,
 Simultaneous packing and covering in Euclidean space,
  Proc. London Math. Soc. 25 (1972) 721--735.

\bibitem{CioabaGIK2021}
  S.M.~Cioab\u{a}, H. Gupta, F. Ihringer, H. Kurihara,
  The least Euclidean distortion constant of a distance-regular graph,
  Discrete Appl. Math. 325 (2023), 212--225.

 \bibitem{Conway1988a}
J.H. Conway, N.J.A. Sloane, Sphere packings, lattices, and groups,
Springer, 1988.

\bibitem{Conway1992}
J.H. Conway, N.J.A. Sloane, Low-Dimensional Lattices VI: Voronoi Reduction of Three-Dimensional Lattices, Proc. R. Soc. Lond. Ser. A Math. Phys. Eng. Sci.
436 (1992), 55--68.

\bibitem{Cucker2001}
F. Cucker, S. Smale,
On the mathematical foundations of learning,
Bull. Amer. Math. Soc. 39 (2001), 1--49.

\bibitem{Ebeling1994a}
  W. Ebeling,
  Lattices and codes,
Friedr. Vieweg \& Sohn, 1994.

\bibitem{Folland1995a}
  G.B. Folland,
  A course in abstract harmonic analysis,
  CRC Press, 1995.

\bibitem{Haviv2010}
  I. Haviv, O. Regev, The Euclidean distortion of flat tori,
 J. Topol. Anal. 5 (2013),  205--223 (extended abstract appeared in
 APPROX/RANDOM 2010).

 \bibitem{Khot2006a}
 S. Khot, A. Naor, Nonembeddability theorems via Fourier analysis,
 Math. Ann. 334 (2006), 821--852 (extended abstract appeared in FOCS 2005)

\bibitem{KobayashiK2015}
T. Kobayashi, T. Kondo,
The Euclidean distortion of generalized polygons,
Adv. Geom. 15 (2015), 499--506.

\bibitem{Kominek2006a}
Z. Kominek, K. Troczka-Pawelec, Some remarks on subquadratic
functions, Demonstratio Math. 39 (2006), 751--758.

\bibitem{LaurentV}
  M. Laurent, F. Vallentin,
  A course on semidefinite optimization,
  Cambridge University Press, in preparation.

\bibitem{Linial1995a}
N. Linial, E. London, Y. Rabinovich, The geometry of graphs and some of its algorithmic applications, Combinatorica 15 (1995), 215--246.

\bibitem{LinialM2000}
  N. Linial, A. Magen,
  Least-distortion Euclidean embeddings of graphs: products of cycles and
  expanders,  J. Combin. Theory Ser. B 79 (2000), 157--171.

\bibitem{LinialMN2002}
  N. Linial, A. Magen, A. Naor,
  Girth and Euclidean distortion,
  Geom. Funct. Anal. 12 (2002), 380--394.

\bibitem{Matousek2002}
J. Matou\v{s}ek, Lectures on discrete geometry,
Springer, 2002.

\bibitem{Mercer1909a}
T. Mercer, Functions of positive and negative type, and their connection
with the theory of integral equations, Transactions London Phil. Soc., (A) 209 (1909), 415--446.

\bibitem{Moore1916a}
E.H. Moore, On properly positive Hermitian matrices,
Bull. Amer. Math. Soc. 23 (1916), 66--67.

\bibitem{Moustrou2023a}
  P. Moustrou, F. Vallentin,
  Least distortion Euclidean embeddings of flat tori.
  Proceedings of the International Symposium on Symbolic \& Algebraic
  Computation (ISSAC 2023, July 24--27, 2023, Troms\o, Norway), 13--23,
  ACM, 2023.

\bibitem{Riesz1990a}
F. Riesz, B. Sz.-Nagy, Functional analysis, Dover Publications, 1990.

\bibitem{Vallentin2008b}
F. Vallentin,
Optimal distortion embeddings of distance regular graphs into
Euclidean spaces,
J. Combin. Theory Ser. B 98 (2008), 95--104.

\bibitem{Venkov2001a} B.B. Venkov, R\'eseaux et designs sph\'eriques,
  pages 10--86 in: R\'eseaux euclidiens, designs sph\'eriques et formes modulaires,
Monogr. Enseign. Math., 37, 2001.

\end{thebibliography}
\end{document}